\documentclass{amsart}

\usepackage{amssymb,amsmath}
\usepackage[neveradjust]{paralist}
\usepackage{epic,eepic}
\usepackage[all]{xy}
\xyoption{dvips}
\usepackage{tikz}
\usetikzlibrary{matrix,arrows,calc}
\usepackage{url}
\usepackage[colorlinks,plainpages,backref]{hyperref}
\usepackage{dsfont}
\renewcommand{\mathbb}{\mathds}
\usepackage{mathrsfs}
\DeclareMathAlphabet{\mathsc}{U}{rsfs}{m}{n}
\renewcommand{\mathcal}{\mathsc}

\theoremstyle{definition}
\newtheorem{ntn}{Notation}[section]
\newtheorem{dfn}[ntn]{Definition}

\theoremstyle{plain}
\newtheorem{lem}[ntn]{Lemma}
\newtheorem{prp}[ntn]{Proposition}
\newtheorem{thm}[ntn]{Theorem}
\newtheorem{cor}[ntn]{Corollary}

\newtheorem{conv}[ntn]{Convention}

\theoremstyle{remark}

\newtheorem{prb}[ntn]{Problem}
\newtheorem{rmk}[ntn]{Remark}
\newtheorem{exa}[ntn]{Example}

\newcommand{\A}{{\mathcal A}}
\newcommand{\calA}{{\mathcal A}}
\newcommand{\calB}{{\mathcal B}}
\newcommand{\calD}{{\mathcal D}}
\newcommand{\calExt}{{\mathcal Ext}}

\newcommand{\calI}{{\mathcal I}}
\newcommand{\calJ}{{\mathcal J}}
\renewcommand{\AA}{{\mathbb A}}

\newcommand{\CC}{{\mathbb C}} 
 
\newcommand{\Z}{{\mathbb Z}}
\newcommand{\PP}{{\mathbb P}}
\newcommand{\KK}{{\mathbb K}}

\newcommand{\NN}{{\mathbb N}}
\newcommand{\calF}{{\mathcal F}}
\newcommand{\calO}{{\mathcal O}}
\newcommand{\RR}{{\mathbb R}}

\newcommand{\p}{{\mathfrak p}}
\newcommand{\del}{\partial}
\newcommand{\set}[1]{\{#1\}}
\newcommand{\abs}[1]{{\left|#1\right|}}
\newcommand{\xymat}{\SelectTips{cm}{}\xymatrix}
\newcommand{\into}{\hookrightarrow}
\newcommand{\onto}{\twoheadrightarrow}
\renewcommand{\to}{\longrightarrow}
\newcommand{\diff}{\,{\mathrm d}}
\newcommand{\ideal}[1]{{\left\langle#1\right\rangle}}
\newcommand{\minus}{\smallsetminus}

\def\bsx{\boldsymbol{x}}
\def\bsc{\boldsymbol{c}}

\DeclareMathOperator{\Der}{Der}

\DeclareMathOperator{\Ext}{Ext}
\DeclareMathOperator{\Hom}{Hom}
\DeclareMathOperator{\supp}{supp}

\DeclareMathOperator{\id}{id}

\DeclareMathOperator{\chr}{char}

\DeclareMathOperator{\codim}{codim}
\DeclareMathOperator{\pd}{pdim}
\DeclareMathOperator{\depth}{depth}

\DeclareMathOperator{\Sing}{Sing}

\DeclareMathOperator{\Syz}{Syz}
\DeclareMathOperator{\Var}{Var}
\DeclareMathOperator{\lcm}{lcm}

\begin{document}

\title[Local cohomology of logarithmic forms]{
Local cohomology of logarithmic forms}

\authors{
\author[G.~Denham]{Graham Denham}
\address{Department of Mathematics, University of Western Ontario, London, Ontario, Canada  N6A 5B7}
\email{\href{mailto:gdenham@uwo.ca}{gdenham@uwo.ca}}
\urladdr{\url{http://www.math.uwo.ca/~gdenham}}
\thanks{G.D.~supported by NSERC}

\author[H.~Schenck]{Hal Schenck}
\address{Department of Mathematics, University of Illinois, Urbana, IL 61801}
\email{\href{mailto:schenck@math.uiuc.edu}{schenck@math.uiuc.edu}}
\urladdr{\url{http://www.math.uiuc.edu/~schenck}}
\thanks{H.S.~supported by NSA H98230-11-1-0170, NSF DMS-1068754}

\author[M.~Schulze]{Mathias Schulze}
\address{Department of Mathematics, University of Kaiserslautern, 67663 Kaiserslautern, Germany}
\email{\href{mailto:mschulze@mathematik.uni-kl.de}{mschulze@mathematik.uni-kl.de}}
\urladdr{\url{http://www.mathematik.uni-kl.de/~mschulze}}
\thanks{}

\author[M.~Wakefield]{Max Wakefield}
\address{Department of Mathematics, United States Naval Academy, Annapolis, MD 21402}
\email{\href{mailto:wakefiel@usna.edu}{wakefiel@usna.edu}}
\urladdr{\url{http://www.usna.edu/Users/math/wakefiel}}
\thanks{M.W.~has been partially supported by the Office of Naval Research.}

\author[U.~Walther]{Uli Walther}
\address{Department of Mathematics, Purdue University, West Lafayette, IN 47907}
\email{\href{mailto:walther@math.purdue.edu}{walther@math.purdue.edu}}
\urladdr{\url{http://www.math.purdue.edu/~walther}}
\thanks{U.W.~supported by NSF grant DMS-0901123}
}
\date{\today}

\begin{abstract}
Let $Y$ be a divisor on a smooth algebraic variety $X$.  We investigate
the geometry of the Jacobian scheme of $Y$, 
homological invariants derived from logarithmic differential forms
along $Y$,
 and their relationship with 
the property that $Y$ be a free divisor. 

We consider arrangements of
hyperplanes as a source of examples and counterexamples.  In particular, we 
make a complete calculation of the local cohomology of logarithmic forms 
of generic hyperplane arrangements.
\end{abstract}



\maketitle
\tableofcontents

\numberwithin{equation}{section}

\section{Introduction}

Positioned in the singularity hierarchy somewhat opposite to a isolated singularity, a free divisor is a reduced hypersurface with large but well-behaved singular locus. 
More precisely, freeness is equivalent to the scheme defined by the Jacobian ideal being empty or Cohen--Macaulay of minimal possible codimension. 
Free divisors occur naturally in deformation theory, as discriminants in
base spaces of versal deformations. 
The classical case, studied by K.~Saito~\cite{Sai80}, is that of a versal deformation of an isolated singularity. 
Later, Looijenga~\cite{Loo84} expanded Saito's ideas to isolated complete intersection singularities; following Bruce~\cite{Bru84} and
Zakalyukin~\cite{Zak83}, Looijenga's results \cite[Cor.~6.13]{Loo84} prove freeness of discriminants of stable mappings $f\colon(\CC^n,0)\to(\CC^p,0)$, $n\ge p$.
For versal deformations of space curves in $3$-space, van~Straten~\cite{vSt95} proved freeness of the reduced discriminant.
Terao~\cite{Ter83} showed that the discriminant of any finite map is free.
Apart from discriminants, there are other natural sources of free divisors in this setup, see for instance \cite{MS10}.

Via deep results of Brieskorn~\cite{Bri71} and Slodowy~\cite{Slo80},
free divisors are linked to representation theory: the discriminants
of ADE singularities are discriminants of finite Coxeter groups with
the same name. 
It turns out that all Coxeter arrangements (even unitary reflection arrangements) as well as their discriminants are free (see \cite{Ter80}).  This led to the study of general free arrangements, and finally to Terao's conjecture, stating that freeness of an arrangement is a combinatorial property. 
It is one of the most prominent open conjectures in the field, and motivated the results in this article.

While reflection groups are discrete, more recently also free divisors
associated with (reductive) algebraic groups have been studied (see
\cite{GMS09,GMNS09,Sev09,dGMS09,GS10}). 
For example, the free divisors associated with a semisimple group are
exactly the free discriminants in Sato--Kimura's classification
of irreducible prehomogeneous vector spaces, and there are exactly four of them up to castling transformations.

There is also a purely ring theoretic version of Saito's theory of free divisors due to Simis~\cite{Sim06}.

In this article, we study homological invariants that stand in the way
of freeness. Our motivation comes from the study of hyperplane
arrangements, but many results are true for more general divisors.

\subsection{Logarithmic forms and vector fields}

Throughout this paper $X$ will be an $\ell$-dimensional smooth algebraic variety
over an algebraically closed field $\KK$.  We shall also be concerned
with affine $\ell$-space 
\[
V=\AA_\KK^\ell
\]
over an arbitrary field $\KK$.  In
both cases, the sheaf $\Omega^1_X$ of differentials on $X$ is locally
free.  By $\calO=\calO_X$ we denote the sheaf of regular functions on
$X$.  We shall freely identify any quasi-coherent sheaf on an affine
scheme with its module of global sections.  In case $X=V$, we pick
coordinates $\bsx=x_1,\ldots,x_\ell$ on $V$, and denote by
\[
S=\KK[V]=\KK[\bsx]
\]
the coordinate ring of $V$.  We shall reserve the symbol 
$R$ to denote arbitrary regular rings.  

In general, choosing a regular system of parameters
$c_1,\ldots,c_\ell$ near $x\in X$ induces an \'etale map $\bsc\colon U
\to \AA^\ell_\KK$ from any open set $U$ on which the differentials of
the $c_i$ are linearly independent (see for example
\cite[VI.1.3]{Bor87}).  By construction, $\bsc$  is the
pullback of $\bsx$ on $\AA^\ell_\KK$ to $U$, and we call it a
\emph{local coordinate system} on $U\subset X$ near $X$; since the map
$\bsc$ is \'etale, the partial differentiation operators $\del_i$ on
$\AA^\ell_\KK$ lift to vector fields $\eta_i$ near $x$ that satisfy
$[\eta_i,c_j]=\delta_{i,j}$. On an analytic (but generally not on
any algebraic) small neighborhood of $x$, $\bsc$ can be made injective.
When we choose $x\in X$ and suitable $U,\bsc,\eta$ in the above
sense, we shall abuse notation and denote $\eta$ by $\del$ and $\bsc$
by $\bsx$.

Let $\Der_X$ be the locally free $\calO$-module of vector fields on
$X$; for $X=V$, its global sections form the module of $\KK$-linear
derivations $\bigoplus_{i=1}^\ell S\cdot\frac{\del}{\del x_i}$ on $S$.
Let $\Omega^\bullet_X$ be the algebraic de Rham complex on $X$; if
$X=V$, then $\Omega^1_X$ is the module with global sections
$\bigoplus_{i=1}^\ell S\diff x_i$. In general,
$\Omega^p_X=\bigwedge^p\Omega^1_X$ is a locally free $\calO_X$-module but
the maps in $\Omega^\bullet_X$ are not $\calO_X$-linear.

Let 
\[
\iota\colon Y\into X
\]
be a reduced divisor on $X$ and denote by $f$ a reduced local defining
equation.  Recall that $\calO_X(kY)$ is the locally free rank-$1$
$\calO_X$-module, locally defined as $\frac1{f^k}\calO_X$. With
$\calO_X(Y)=\calO_X(1Y)$, for any sheaf $\calF$ on $X$, let
$\calF(Y)=\calF\otimes_{\calO}\calO_X(Y)$.  On the other hand
$\calO_X(*Y)$ is by definition 
$j_*\calO_{X\smallsetminus Y}$, where $j\colon X-Y\to X$ is the open embedding.
Our main focus is on the following two objects derived from $Y$:

\begin{dfn}
The sheaf of \emph{logarithmic vector fields} along $Y$ is the $\calO$-module
\[
\Der_X(-\log Y)=\{\theta\in\Der_X\mid\theta(\calO(-Y))\subseteq\calO(-Y)\}.
\]
The $\calO$-module complex $\Omega^\bullet_X(\log Y)$ of
\emph{logarithmic differential forms} along $Y$ is the largest
subcomplex of $\Omega^\bullet_X(*Y)$ contained in $\Omega^\bullet_X(Y)$:
\[
\Omega^p_X(\log Y)=\{\omega\in\Omega^p(Y)\mid\diff\omega\in\Omega^{p+1}(Y)\}.
\]
Geometrically, vector fields in $\Der(-\log Y)$ are tangent to the smooth locus of $Y$.  
\end{dfn}

\begin{conv}
If $X$ is understood from the context, we suppress it in the
subscripts in $\calO_X$, $\Omega^p_X$, $\Der_X(-\log Y)$ and
$\Omega^p_X(\log Y)$.
\end{conv}

Obviously, $\Omega^p(\log Y)\not=0$ only for $0\le p\le \ell$, and one has
$\Omega^0(\log Y)=\calO$ and $\Omega^\ell(\log Y)=\Omega^\ell(Y)$. If
$X=V$ then $\Omega^\ell(\log Y)=\frac1fS\diff\bsx$ where $\diff\bsx=\diff
x_1\wedge\ldots\wedge\diff x_\ell$.

In \cite{DS10}, an alternative definition is used for $\Der(-\log Y)$
and $\Omega^\bullet_V(\log Y)$. Our definition agrees with theirs in
the arrangement case, but works more generally.  By \cite[\S2]{DS10},
$\Der(-\log Y)$ and $\Omega^\bullet_V(\log Y)$ are reflexive and hence
$Y$-normal in the sense of \cite{DS10}. 
This implies that certain properties of $Y$, obviously valid at all smooth
points, are retained at the singular points of $Y$.  For instance:
\begin{enumerate}
\item $\Omega^\bullet(\log Y)$ is stable under contraction against elements of $\Der(-\log Y)$.
\item Contraction sets up a perfect pairing
\[
\Der(-\log Y)\times\Omega^1(\log Y)\to\calO(Y)
\]
and an identification
\[
\Der(-\log Y)=\Omega^{\ell-1}(\log Y);
\]
for $X=V$, these are defined by $(\sum_i\eta_i\,\diff x_i,\sum_i\theta_i\,\frac{\del}{\del x_i})\mapsto\sum_i\theta_i\eta_i$ and $\sum_i\eta_i\frac{\del}{\del x_i}\leftrightarrow\sum_i(-1)^i\frac{\eta_i}{f}\diff x_1\wedge\ldots\wedge\widehat{\diff x_i}\wedge\ldots\wedge\diff x_\ell$ respectively.
\item The exterior product induces a perfect pairing
\begin{equation}\label{omega-products}
\Omega^p(\log Y)\times \Omega^{\ell-p}(\log Y)\to \Omega^\ell(\log Y).
\end{equation}
\item \label{rem:product}
The natural map $
j_p\colon \bigwedge^p\Omega^1(\log Y)\to \Omega^p(\log Y)
$ is injective.
\end{enumerate}

\begin{dfn}
A \emph{free divisor} is a divisor $Y$ for which $\Omega^1(\log Y)$ is
a free module.  As it turns out, for free $Y$, the modules $\Der(-\log
Y)$ and $\Omega^p(\log Y)$, $0\le p\le \ell$, are all free as well and
$\Omega^p(\log Y)=\bigwedge^p\Omega^1(\log Y)$.  The \emph{free locus}
of the divisor $Y$ is the set of points $x\in Y$ where $\Omega^1(\log
Y)$ is a locally free $\calO$-module; it includes the complement of
the singular locus of $Y$ in $X$.
\end{dfn}

\subsection{Euler homogeneity and Jacobians}\label{sec:EH}

Let $U\subseteq X$ be an affine open subset and choose a local defining equation $f$, defined up to units, of $U\cap Y$ relative to some local coordinate system $\bsx$ on $U$. 
While the particular choice of the coordinate system is relevant, we hide its presence in the symbol ``$U$''.

We denote
\[
\calJ_{U,f}=\calO_U\langle\frac{\del f}{\del x_1},\dots,\frac{\del f}{\del
    x_\ell}\rangle
\]
the $\calO_U$-ideal generated by the partial derivatives of $f$. 
Note that $\calJ_{U,f}$ varies with $f$ even if $U$, $Y$ and $\bsx$ are fixed. 
(Example: $f_1=x-1$ and $f_2=x^2-1$ on $\CC^1\smallsetminus\Var(x+1)$).
The $1$-st Fitting ideal of $\Omega^1_{U\cap Y}$ is 
\[
\calJ_{U,Y}=\calJ_{U,f}\cdot\calO_{U\cap Y}.
\] 
In contrast to $\calJ_{U,f}$, $\calJ_{U,Y}$ only depends on $U$ and $Y$, and the various $J_{U,Y}$ patch to an ideal sheaf $\calJ_Y$ on $Y$. 
We call the subscheme
\[
Z=\Var(\calJ_Y)\subseteq Y
\]
with structure sheaf $\calO_Y/\calJ_Y$ the \emph{Jacobian scheme} of $Y$.
We shall use $\calI_{X,Z}$ for the preimage ideal sheaf of $\calJ_Y$
under the natural projection 
\[
\iota^\#\colon\calO_X\onto\iota_*\calO_Y.
\]

Write $\Syz \calJ_{U,f}$ for the syzygy sheaf of $\calJ_{U,f}$. We
freely use the obvious identification between this 
sheaf with the vector fields on $U$ that annihilate
$f$.  Geometrically, its elements are vector
fields tangent to the smooth part of all level sets of $f$.  
Like $\calJ_{U,f}$, $\Syz \calJ_{U,f}$ varies with $f$, even if
$U$ and $\bsx$ are fixed.  For an
arrangement $Y=\A$ in $X=V$, this definition agrees with
$D^{(1,\dots,1)}(\A)$ from \cite{DS10}: we develop this further in
\S\ref{ss:arr}.

There is a commutative diagram in the category of $\calO_X$-modules
\begin{equation}\label{eq:logflogY}
\xymat{
0\ar[r]&\Der_{U}(-\log Y)\ar[r]&\Der_U\ar[r]^-{\bullet(f)}&\iota_*(\calO_{Y\cap
  U})\ar[r]&\iota_*(\calO_{Y\cap U}/\calJ_{U,Y})\ar[r]&0\\
0\ar[r]&\Syz\calJ_{U,f}\ar[r]\ar@{^(->}[u]&\Der_U\ar[r]^-{\bullet(f)}\ar@{=}[u]&\calO_U\ar[r]\ar@{->>}[u]&\calO_U/\calJ_{U,f}\ar[r]\ar@{->>}[u]&0\\
}
\end{equation}
with exact rows. Here, $\bullet(f)$ denotes application to $f$.  

\begin{dfn}\label{dfn-Euler}
We say that $Y$ admits an \emph{Euler vector field} $\chi\in\Der_U$ on
the open set $U\subseteq X$ if $\chi(f)=f$ for some local
reduced defining equation $f\in\calO_X(U)$ of $Y\cap U$.
If $Y$ can be covered by such open sets $U$, we call $Y$ \emph{Euler homogeneous}.
\end{dfn}

Let us suppose now that $Y$ admits a global Euler vector field $\chi$ on $X$. 
This has two important consequences. 
Firstly, for any local defining equation $f$ on any open affine $U$, $f$ is an element of
$\calJ_{U,f}$.
It follows that $\calJ_{U,f}=\calI_{X,Z}$ depends only on $U$ and $Y$
but not on the choice of $f\in\calO_X(U)$ or the coordinate system
$\bsx$; in particular, 
one can extend diagram \eqref{eq:logflogY} to all of $X$. 
Secondly, the map
\[
\Der_U(-\log Y)\ni\delta\mapsto
\delta-\frac{\delta(f)}{f}\cdot\chi\in\Syz\calJ_{U,f} 
\]
defines a non-canonical local splitting of the vertical inclusion in
\eqref{eq:logflogY}.  Thus, there is a split exact sequence of sheaves
\begin{equation}\label{eq:split}
\xymat{
0\ar[r] & \Syz\calJ_{U,f}\ar[r] & \Der_U(-\log Y)\ar[r] & \calO_U\cdot\chi\ar[r] & 0.
}
\end{equation}
For an arrangement $Y=\A$ in $X=V$, $\Der(-\log Y)/\calO\cdot\chi=D_0(\A)$ in the notation of \cite{DS10}.

\subsection{Arrangements}\label{ss:arr}

We introduce the family of divisors that we are most interested in
and refer to \cite{OTbook} for further details.

\begin{ntn}
By an \emph{arrangement} we mean a finite union $\calA$ of hyperplanes in
affine space. 
Throughout, we assume that arrangements are
\emph{central}, involving $n$ hyperplanes in $V$ passing through the origin. 
It is easy to see that $D(\calA)$ and $\Omega^p(\calA)$ are graded modules in this case. 
Abusing notation, we view $\calA$ both as a divisor $Y=\bigcup_{H\in\calA}H$, and as a list of hyperplanes.

For $H\in\calA$, let $\alpha_H\in V^*$ be a defining linear form.
Then  
\[
f=\prod_{H\in\calA}\alpha_H
\]
is a defining equation of $\A$ on $V$, and we call
\[
J_\A=\ideal{\frac{\del f}{\del x_1},\dots,\frac{\del f}{\del x_\ell}}\subseteq S
\]
the \emph{Jacobian ideal} of $\A$.  In the case where
$Y=\calA$ is an arrangement, in order to streamline the notation and
following an established convention, we write $D(\calA)$ for $\Der(-\log
Y)$, and $\Omega^p(\calA)$ for $\Omega^p(\log Y)$.
\end{ntn}

\begin{rmk}\label{rmk:projA}
Since  $\A$ is a cone over the origin, we can consider its image $\PP\A$
in $\PP^{\ell-1}_\KK$. 
By choosing a hyperplane $H\in\A$, we choose an affine chart in $\PP^{\ell-1}_\KK$
given by $\alpha_H\neq0$. 
The restriction of $\PP\A$ to such a chart is the noncentral arrangement of hyperplanes in $\AA^{\ell-1}$
given by dehomogenizing $f$ relative to $H$.  By reversing this
construction, many questions regarding general hyperplane arrangements can be
reduced to the case of central arrangements (see, e.g., \cite[\S3.2]{OTbook}).
\end{rmk}

Recall that a flat $F$ of an arrangement $\A$ is an intersection of 
hyperplanes.  The intersection
lattice $L(\A)$ is the set of all flats, partially ordered by reverse
inclusion.

\begin{prp}\label{prp:arrEuler}
Let $p=\chr\KK$.
A hyperplane arrangement $\A$ is an Euler-homogeneous divisor provided
that $p$ does not divide
\begin{equation}\label{eq:goodp}
\lcm_{F\in L(\A)}\abs{\set{H\in\A\colon H\leq F}}.
\end{equation}
\end{prp}

\begin{proof}
For a point $x\in \A$, let $F_x$ be the set-theoretically smallest flat
containing $x$.  Let $k$ be the number of hyperplanes containing $F_x$.
The divisor is defined locally by a homogeneous
polynomial of degree $k$, so the derivation $1/k\sum_{i=1}^\ell x_i 
\partial/\partial x_i$ is an Euler vector field in the neighborhood of
$x$, provided that $p\nmid k$.
\end{proof}

One then naturally arrives at:

\begin{dfn}\label{dfn:goodp}
If $\A$ is an arrangement over a field $\KK$, say $p=\chr \KK$ is {\em good} for $\A$ if $p$ does not divide the expression \eqref{eq:goodp}.
\end{dfn}

\begin{dfn}
If the complement in $X$ of the free locus of $Y$ is zero-dimensional,
we say that \emph{$Y$ is free outside points}. If $X=\AA^\ell_\KK$,
and if 
$Y$ is quasi-homogeneous and free outside points then the origin is
the only non-free point of $Y$ and we then call $Y$ \emph{free outside $0$}.
\end{dfn}

\begin{rmk}
If $Y$ is a central hyperplane arrangement in affine space, then the
traditional term for $Y$ being free outside points is ``locally
free''. 
This terminology is traditional in arrangement theory, and is 
rooted in Remark~\ref{rmk:projA}: 
for any homogeneous $Y\subset V$, the sheaf $\Der(-\log\PP Y)$ on $\PP^{\ell-1}_\KK$ is locally free if and only if $Y$ is free outside the origin of $\AA^\ell_\KK$ in the above sense.
\end{rmk}

In Section~\ref{sec-logforms}, we investigate general homological
properties of $\Omega^\bullet(\log Y)$.  One natural such property, a
relaxation of freeness and inspired by the arrangement case, is that
of tameness:

\begin{dfn}\label{df:tame}
The divisor $Y\subset X$ is called \emph{tame} if $\pd_{\calO_{X,x}}\Omega^p(\log Y)\le p$ for all $x\in X$ and all $0\le p\le \ell$. 
\end{dfn}

Many interesting families of hyperplane arrangements are tame; these
include all arrangements in $\AA^3$; generic arrangements;
supersolvable and reflection arrangements (the last two of which are,
in fact, free).  The tame hypothesis appears explicitly first in \cite{RT91} and
frequently since.  In particular, tame arrangements satisfy a
Logarithmic Comparison Theorem~\cite{WiYu97}.  More discussion and
another application is found in \cite{CDFV09}.

We introduce in Section~\ref{sec-wakefield} a weaker version of
tameness and, with the general results in Section~\ref{sec-logforms},
state a criterion that forces an Euler-homogeneous divisor $Y$ to be
free provided it is free outside points and has a certain tameness property. 

In Section~\ref{sec-generic}, we investigate the case of a generic central arrangement $\calA$ and determine formulas describing the Hilbert function of the modules $\Ext^i_S(\Omega^p(\calA),S)$ for any $p$ and $i$. 

Our hope is that this article serves to trigger more studies on the interplay of homological properties of the logarithmic $p$-forms of arrangements and their combinatorics. 

\section{Embedded primes of the Jacobian ideal}\label{sec-wakefield}

Our hypotheses on $Y$, when combined with the Jacobian criterion for smoothness, imply that $\calI_{X,Z}$ has codimension at least $2$.
If the divisor $Y$ is not smooth, the upper row of display
\eqref{eq:logflogY} shows that $\Der(-\log Y)$ is a vector bundle if
and only if $\calI_{X,Z}$ is a Cohen--Macaulay ideal of codimension two.  For
this to happen, $\calI_{X,Z}$ must not have embedded primes.  In this
section, we study the converse of this implication for $Y$ that is
free outside points under two additional hypotheses: Euler
homogeneity, and a weakened form of tameness.

\subsection{The algebraic case}

For an Euler homogeneous divisor $Y$, we obtain from \eqref{eq:logflogY} the following identifications if $i > 0$:
\begin{equation}\label{eq:Ext-shift}
\calExt^i_\calO(\Der(-\log Y),\calO)\cong
\calExt^i_\calO(\Syz\calJ_{U,f},\calO)\cong
\calExt^{i+2}_\calO(\calO_Z,\calO).
\end{equation}
\begin{lem}\label{lem:jac}
The following are equivalent:
\pushQED{\qed}
\begin{enumerate}
\item $Z$ is of pure codimension $2$ in $X$.
\item $\codim_X\supp\calExt^p_\calO(\calO_Z,\calO)>p$, for $p\geq3$.
\end{enumerate}
If $Y$ is Euler homogeneous then the above conditions are equivalent to:
\begin{enumerate}[(3)]
\item $\codim_X\supp\calExt^p_\calO(\Der(-\log Y),\calO)>p+2$, for $p\geq 1$.\qedhere
\end{enumerate}
\end{lem}

\begin{proof}
By \cite[Thm.~1.1]{EHV92}, a prime ideal of codimension $p+2$ is associated to $\calI_{X,Z}$ if and only if it is contained in the support of 
$\calExt^{p+2}_\calO(\calO_Z,\calO) = \calExt^p_\calO(\Der(-\log Y),\calO)$.
\end{proof}

We can express freeness outside points homologically in terms of
$\calExt$-sheaves as follows.

\begin{lem}\label{lem:locfree}
The following are equivalent:
\begin{enumerate}
\item $Y$ is free outside points.
\item $\calExt^p_\calO(\calExt^q_\calO(\Omega^1(\log Y),\calO),\calO)=0$ if both  $q>0$ and $p<\ell$.
\end{enumerate}
\end{lem}

\begin{proof}
The statement can be verified locally at a point $x\in X$.
So we may replace $Y$ by its germ at $x$ and $\calO$ by $R=\calO_x$.
If $Y$ is free outside points but not free, then all associated primes of $\Ext^q_R(\Omega^1(\log Y),R)$ are maximal ideals for $q>0$. 
Then by Ischebeck's Theorem (see \cite[Thm.~17.1]{Mat89}), only $\Ext^\ell_R(\Ext^q_R(\Omega^1(\log
Y),R),R)$ can be non-zero. 
Conversely, if $Y$ is not free outside points, then there is a minimal prime $\p$ of codimension $p<\ell$, and a $q>0$ for which $\Ext^q_R(\Omega^1(\log Y),R)_\p\ne0$. 
By \cite[Thm.~1.1]{EHV92}, $\p$ is then associated to $\Ext^p_R(\Ext^q_R(\Omega^1(\log Y),R),R)$ which cannot be zero in this case.
\end{proof}

In order to relate the Ext-modules in Lemmas~\ref{lem:jac} and \ref{lem:locfree}, we use a double-dual spectral sequence.

\begin{rmk}\label{rem-Ext-ss}
The following construction, a Grothendieck spectral sequence obtained by composing the functor $\Hom_R(-,R)$ with itself, appears already in \cite{Ro62}. 
Suppose $R$ is a regular ring.
Let $0\to F_\bullet\to M\to 0$ be a projective resolution of $M$ with dual
complex $\Hom_R(F_\bullet,R)=(0\to F_0^\vee\to\cdots\to F_d^\vee\to
0)$. 
Let $C_{\bullet,\bullet}$ be a projective Cartan--Eilenberg resolution of $F_\bullet^\vee$. 
Consider the two usual spectral sequences to the dual double complex $E^0_{\bullet,\bullet}=\Hom_R(C_{\bullet,\bullet},R)$.
The total complex of $E$ is a representative of $\RR\Hom_R(\RR\Hom_R(M,R),R)\cong M$. Thus, the spectral sequences converge to $M$.

The second spectral sequence has, with suitable choice of indices, the differentials $d^0$ oriented downward and $d^1$ oriented to the left, and   
\begin{equation}\label{eq:E2SS}
E_{p,q}^2=\Ext_R^{-p}(\Ext_R^q(M,R),R)\Longrightarrow M.
\end{equation}
With this choice, the abutment $M$ is regarded as a chain complex concentrated in total degree $0$.  
\end{rmk}

\begin{prp}\label{prop-ext}
If $Y$ is free outside points, then there are isomorphisms 
\[
\calExt^{\ell-p-1}_\calO(\Der(-\log Y),\calO)\to
\calExt^\ell_\calO(\calExt^p_\calO(\Omega^1(\log Y),\calO),\calO),\quad p>0.
\]
\end{prp}

\begin{proof}\pushQED{\qed}
As we will see below, over any affine open subset $U\subset X$, these isomorphisms are morphisms in the Grothendieck spectral sequence in Remark~\ref{rem-Ext-ss} for $R=\calO(U)$.
By functoriality of this spectral sequence, these patch together to an isomorphism defined on $X$.

The spectral sequence $\eqref{eq:E2SS}$, applied to $M=\Omega^1(\log Y)(U)$, is illustrated in Figure~\ref{fig:SS}.
\begin{figure}[h]
\caption{$E^2_{pq}$}\label{fig:SS}
\begin{tikzpicture}[=>stealth]
\matrix(m)[matrix of math nodes,row sep=1em,column sep=1em,text width=3em,text height=2ex,text centered,nodes in empty cells]
{
{*} & 0 &\cdots & 0 & 0 & 0 & \ell-2 \\
\vdots & \vdots & & \vdots & \vdots & \vdots & \vdots \\
{*} & 0 & \cdots & 0 & 0 & 0 & 1\\
0 & 0 & {*} & \cdots & {*} & \Omega^1(\log Y) & 0\\
-\ell & -(\ell-1) & -(\ell-2) & \cdots & -1 & 0 & \null \\
};
\draw[->] (m-4-5) edge (m-1-1);
\draw[->] (m-4-3) edge (m-3-1);
\draw (m-1-1.north west) -- (m-1-7.north east);
\draw (m-5-1.north west) -- (m-5-7.north east);
\draw (m-5-1.south west) -- (m-5-7.south east);
\draw (m-1-1.north west) -- (m-5-1.south west);
\draw (m-1-7.north west) -- (m-5-7.south west);
\draw (m-1-7.north east) -- (m-5-7.south east);
\node at ($(m-5-7.north east)+(-0.2,-0.2)$) {q};
\node at ($(m-5-7.south west)+(0.2,0.2)$) {p};
\draw[=>stealth] (m-5-7.north west) -- (m-5-7.south east);
\end{tikzpicture}
\end{figure}
By Lemma~\ref{lem:locfree}, the only non-zero entries are in column
$p=-\ell$, or in row $q=0$. Let $R=\calO(U)$, and omit the argument
``$U$'' in all instances of global sections below.  
For $q=0$, we have $E^2_{-p,0}=\Ext_R^p(\Der(-\log Y),R)$, which is non-zero only if $p\le\ell-2$ since $\depth \Der(-\log Y)\ge2$ by reflexivity of $\Der(-\log Y)$. 
As $\Omega^1(\log Y)$ is reflexive, $E^2_{0,0}=\Omega^1(\log Y)$.
Hence, the only possible non-zero higher differentials $d^{p+1}\colon E^{p+1}_{-\ell+p+1,0}\to E^{p+1}_{-\ell,p}$ must be isomorphisms
\[
d^r\colon \Ext^{\ell-p-1}_R(\Der(-\log Y),R)\to
\Ext^\ell_R(\Ext^{p}_R(\Omega^1(\log Y),R),R). \qedhere
\]
\end{proof}

Adding a weakened form of tameness as a second additional hypothesis, we obtain the following result:

\begin{thm}\label{th:wakefield}
Suppose that $Y$ is free outside points and Euler homogeneous, and that
$\pd_{\calO_x}\Omega^1(\log Y)_x\leq1$ for all $x\in Y$. 
Then $Y$ is free if and only if $Z$ is pure of codimension $2$ in $X$.
\end{thm}

\begin{proof}
All plane curves are free, so assume $\ell\geq3$.  

The statement is clearly local, so we may replace $Z\subseteq Y\subseteq X$ by their germs at $x\in Z$ and work over $R=\calO_x$.
As $\pd_R\Omega^1(\log Y)\leq1$, we have $E^2_{p,q}=0$ for $q\geq2$ in \eqref{eq:E2SS} for $M=\Omega^1(\log Y)_x$. 
By Proposition~\ref{prop-ext}, $\Ext^p_R(\Der(-\log Y),R)=0$ unless $p=0,\ell-2$.

Since $Z$ is pure of codimension $2$ in $X$, Lemma~\ref{lem:jac} implies 
\[
\codim_X\supp\Ext^\ell_R(\Der(-\log Y),R)>\ell=\dim R.
\]
Therefore, $\Ext^\ell_R(\Der(-\log Y),R)=0$ for all $p>0$ and hence $\Der(-\log Y)$ is free as claimed.
\end{proof}

\subsection{Arrangements}

By Proposition~\ref{prp:arrEuler}, hyperplane arrangements in good characteristic provide a family of examples of Euler homogeneous divisors.
In this case we can allow a larger non-free locus.

\begin{ntn}
If $F$ is a flat of an arrangement, then  $\calA_F$ is the arrangement of all hyperplanes containing $F$. 
\end{ntn}

\begin{cor}\label{cor-max}
Let $\calA$ be a hyperplane arrangement in good characteristic with $\pd_S\Omega^1(\calA)\le 1$. 
If $J_\A$ is pure of codimension $2$ then $\calA$ is free.
\end{cor}

\begin{proof}
As the statement is local, we may assume that $\A$ is central.
Let $X_0\in L(\A)$ be a flat such that $\Omega^1(\A_{X'})$ is free for
all flats $X'\supsetneq X_0$;
then it suffices to prove that $\Omega^1(\A_{X_0})$ is free.

Both $\A\mapsto\Omega^1(\A)$ and $\A\mapsto J_\A$ are local functors in the sense of \cite{ST87}, so we may replace $\A$ by $\A_{X_0}$.

Now $\A$ is the product of its essentialization $\A_0$ (in $X/X_0$)
and the affine space $X_0$.  Then $\Omega^1(\A)$ is the
direct sum of the free $S$-modules $\calO_{X/X_0}(X/X_0)\otimes_\KK
\Omega^1_{X_0}$ and $\calO_{X_0}(X_0)\otimes_\KK \Omega^1_{X/X_0}(\A_0)$,
these being $S$-modules via $S\onto \calO_{X_0}(X_0)$ and
$S\onto\calO_{X/X_0}(X/X_0)$. 
Moreover, $J_\A$ and $J_{\A_0}$ have the same generators in a generic
point of $X$. 
We may therefore replace $\A$ by $\A_0$.

Now $\A$ is free outside points and hence free by
Theorem~\ref{th:wakefield}.
\end{proof}

\begin{rmk}\
\begin{asparaenum}

\item 
The hypothesis that $\pd_S\Omega^1\leq 1$ is necessary. 
The Edelman--Reiner example \cite{ER93} defined by
\begin{equation}\label{eq:ER}
Q=\prod_{\alpha\in\set{0,1}^4,\alpha\neq(0,0,0,0)} (\alpha_1x_1+\alpha_2x_2+
\alpha_3x_3+\alpha_4x_4)
\end{equation}
is free outside the origin, its Jacobian ideal has no embedded primes,
yet it is not free.  Here, $\pd_S\Omega^1(\calA)=2$.

\item Add a generic hyperplane $x_5=0$ to the example $\calA$ above to
  obtain an arrangement $\calB$ defined by $Qx_5$.  The associated
  primes of $J(Qx_5)=(Q,x_5J(Q))=(Q,x_5)\cap J(Q)$ are all codimension
  $2$, since $J(Q)$ has no submaximal embedded components.
  Nevertheless, $\calB$ is not free outside points, since $\calA$
  defines a non-maximal, non-free flat of $\calB$.  Again, $\pd
  \Omega^1(\calB)=2$.

\item One way of reading (the proof of) Corollary~\ref{cor-max} is:
  ``if $\pd_S\Omega^1(\calA)\le 1$ then the top-dimensional non-free
  flats of $\calA$ are the minimal embedded primes of $J_\calA$.''

\item Let $\calA$ be of rank three.  Then $\calA$ is free outside
  points and $\pd_S\Omega^1\leq1$ due to reflexivity.  Thus, in this
  case Theorem~\ref{th:wakefield} becomes the well-known statement
  ``$\calA$ is free if and only if $J_\calA$ has no embedded primes.''

\item
The hypothesis that $J_\calA$ has no embedded primes is necessary.
Generic arrangements of $n>\ell$ hyperplanes are free outside points. 
If $\ell\geq3$, they have $\pd_S\Omega^1(\calA)=1$ by \cite[Cor.~7.7]{Zie89}, so they are not free.

\end{asparaenum}
\end{rmk}

\begin{rmk}
In \cite{MuSc01}, Proposition 2.10 and 2.7, it is shown that if $\dim X\geq p-1$ then freeness of $D^p_{\calA_X}$ is equivalent to freeness of $\calA_X$.
\end{rmk}

\subsection{Analytic divisors}

We aim for a generalization of Corollary~\ref{cor-max} to more general
divisors.  For the desired analogue of the stratification by flats of
an arrangement we need to move to the analytic category.  So let $X$ be a
complex $\ell$-dimensional analytic manifold.

The stratification by flats for arrangements is a special case of the
\emph{logarithmic stratification} $\{Y_\alpha\}$ along $Y$ introduced by
K.~Saito \cite[\S3]{Sai80}.
It is uniquely characterized by the following properties \cite[Lemma~3.2]{Sai80}:
\begin{enumerate}
\item Each stratum $Y_\alpha$ is a smooth connected immersed submanifold of $X$, and $X$ is their disjoint union.
\item If $x\in Y_\alpha$ then the tangent space $T_{Y_\alpha,x}$ coincides with the space $\Der(-\log Y)_x$ of logarithmic vector fields evaluated at $x$.
\item If $Y_\alpha$ meets the closure of $Y_\beta$ in $X$ with $\alpha\not=\beta$ then $Y_\alpha$ is in the boundary of $Y_\beta$. 
\end{enumerate}

Note that $Y$, $X\smallsetminus Y$, and $\Sing(Y)$ are (unions of)
logarithmic strata.
The term ``logarithmic stratification'' is a misnomer as it is not locally finite in general.

\begin{exa}[{\cite[Rem.~4.2.4]{Cal99}}]
Each point on the $z$-axis is a logarithmic stratum of the free divisor $Y=(xy(x+y)(x+zy))$.
\end{exa}

\begin{dfn}[{\cite[Def.~3.8]{Sai80}}]
The divisor $Y$ is called \emph{holonomic} (at $x\in X$), if its logarithmic stratification is locally finite (at $x$). 
For free $Y$, holonomicity is also referred to as \emph{Koszul freeness}.
\end{dfn}
 
By \cite[(3.10.i),(3.13.(ii)]{Sai80}, if $Y$ is holonomic at $x\in Y_\alpha$ then the manifold topology of $Y_\alpha$ is the induced topology from $X$ near $x$; this can fail in general.

Saito introduced the \emph{logarithmic characteristic variety} $L_Y(\log Y)\subset T^*_X$, defined by the symbols of $\Der(-\log Y)$ with respect to the order filtration $F_\bullet$ on $\calD_X$.
He showed that holonomicity is (locally) equivalent to $L_Y(-\log Y)$ having minimal dimension $\ell$ \cite[(3.18)]{Sai80}.
In particular, Saito's holonomicity implies that $M^{\log
  Y}=\calD_X/\calD_X\cdot\Der(-\log Y)$ is a holonomic $D$-module in
the sense of Kashiwara. This implication cannot be reversed, as we were
informed by F.~Castro Jim\'enez:

\begin{exa}\label{exa-paco}
The divisor $Y$ given by $(xz+y)(x^4+y^5+xy^4)$ is not holonomic in the sense of Saito, but $M^{\log Y}$ is a holonomic $D$-module.
\end{exa}

By \cite[(3.6)]{Sai80}, the logarithmic stratification provides us with the local product structures that we used in the proof of Corollary~\ref{cor-max}:

\begin{prp}\label{prop:loc-prod}
Let $x\in Y_\alpha$ and set $m=\dim Y_\alpha$. 
On small neighborhoods $U$ of $x$, $Y\cap
U\subseteq U$ is defined by some $f$ such that
$f(z)=f(z_1,\ldots,z_{\ell-m},0,\dots,0)$ and $Y_\alpha\cap
U_x=\{z_1=\dots=z_{\ell-m}=0\}$ for some coordinate system
$z_1,\dots,z_\ell$ on $U$.  In other words, there is an isomorphism of
pairs of germs
\begin{equation}\label{eq:loc-prod}
(Y_x,X_x)\cong(Y'_0,\CC^{\ell-m}_0)\times(Y_{\alpha,x},Y_{\alpha,x}),\quad Y_{\alpha,x}\cong\CC^m_0.
\end{equation}
This isomorphism identifies the logarithmic stratifications of $Y_x$ and $Y'_0$.\qed
\end{prp}

By connectedness of the strata this implies in particular:

\begin{cor}
Freeness of $Y$ at $x\in Y_\alpha$ is a property of $Y_\alpha$.
\end{cor}
There is a similar corollary for holonomicity, and the notion of a holonomic stratum, but we do not need it (see \cite[(3.10) Prop.~i)]{Sai80}).

\begin{dfn}
We call $Y_\alpha$ a free stratum if $Y$ is free at some $x\in Y_\alpha$.
\end{dfn}

Euler homogeneity does not descend in general to factors in a  product;
in fact, products with analytic affine spaces are always Euler homogeneous.

\begin{exa}
Let $Y=Y'\times\AA^1_\CC$ with $Y'=(f(x))\subseteq\AA^{\ell -1}_\CC$.
Then $e^y\cdot f(x)$ is a defining equation for $Y$ and $\chi=\frac{\del}{\del y}$ is an Euler vector field for $Y$.
So $Y$ is Euler homogeneous. 
However, choosing $Y'$ with an isolated non-quasihomogeneous singularity, it is not Euler homogeneous by Saito's theorem \cite{Sai71}.
\end{exa}

We need the following stronger version of Euler homogeneity which is well-known in the context of the Logarithmic Comparison Theorem for free divisors (see \cite[Conj.~1.4]{CMNC02} and, for example, \cite{GS06,GMS09,GS10}).

\begin{dfn}
The divisor $Y$ is called \emph{strongly Euler homogeneous} if it admits, at each point $x\in X$, an Euler vector field vanishing at $x$.
\end{dfn}

By \cite[Lem.~3.2]{GS06}, strong Euler homogeneity for $Y$ and $Y'$ in \eqref{eq:loc-prod} are equivalent.

\begin{thm}\label{thm-max-anal}
Let $Y$ be a strongly Euler homogeneous and holonomic divisor in a complex manifold $X$ such that $\pd_{\calO_x}\Omega^1(\log Y)_x\le 1$ for all $x\in X$.
Then $Y$ is free if $Z$ is of pure of codimension $2$ in $X$.
\end{thm}

\begin{proof}
We proceed along the lines of the proof of Corollary~\ref{cor-max}.

Let $Y_\alpha$ be a stratum such that all strata $Y_\beta$ containing
$Y_\alpha$ in their closure are free.
To show that $Y_\alpha$ is free, we may replace $Z\subset Y\subset X$ by their germs at $x\in Z$ and work over $R=\calO_x$.

Using Proposition~\ref{prop:loc-prod}, we may replace $Y$ by $Y'$ in
\eqref{eq:loc-prod}, and hence assume that $Y$ is free outside points.
Here we use that strong Euler homogeneity is preserved by \cite[Lem.~3.2]{GS06}, freeness and $\pd_R\Omega^1(\log Y)$ are preserved by \cite[Lem.~2.2.(iv)]{CNM96}, and $Z=Z'\times Y_{\alpha,x}$ where $Z'$ is the Jacobian scheme of $Y'$ in \eqref{eq:loc-prod}. 

Finally, we apply Theorem~\ref{th:wakefield} whose proof works also in the analytic setup.
Thus, $Y_\alpha$ is free and the claim then follows by descending induction on $\dim Y_\alpha$.
\end{proof}

\section{Homological properties of logarithmic forms}\label{sec-logforms}

The previous section indicates the importance of the condition $\pd_\calO\Omega^1(\log Y)\le 1$. 
The results in the present section come from an attempt to understand what can be rescued if this hypothesis on $\pd_\calO\Omega^1(\log Y)$ is false or unknown.
Specifically, we are interested in the difference between $\Omega^p(\log Y)$ and $\bigwedge^p\Omega^1(\log Y)$. 
Recall that $M^\vee=\Hom_R(M,R)$ for any module.

\begin{conv}
Throughout this section we assume that, if $\A$ is an arrangement of rank
$\ell$ over a field $\KK$, then $\chr\KK$ is either zero, 
or both good (Definition~\ref{dfn:goodp}) and at least $\ell$.
\end{conv}

\subsection{Higher forms as exterior products}

The exterior product gives an exact sequence
\begin{equation}\label{eq:E}
\xymat{ 
\bigwedge^p\Omega^1(\log Y)\ar[r]^-{j_p} &
\Omega^p(\log Y)\ar[r] & E^p_Y:=\displaystyle
\frac{\Omega^p(\log Y)}{\bigwedge^p\Omega^1(\log Y)} \ar[r] & 0. 
}
\end{equation}
Saito~\cite{Sai80} showed that in all free points of $Y$, $j_p$ is an
isomorphism for all $p$; see also~\cite[formula (2.3)]{DS10}.  In
general, one knows from \cite[Prop.~2.2]{DS10} that $\Omega^p(\log Y)$
is the reflexive hull $\big(\bigwedge^p\Omega^1(\log
Y)\big)^{\vee\vee}$ of $\bigwedge^p\Omega^1(\log Y)$.

Given some information about the codimension of the
non-free locus, it is possible to say more.  Musta\c ta and
Schenck prove the following for arrangements in \cite{MuSc01}:

\begin{thm}\label{th:MSwedges}
Let $\calA$ be free outside points and $\pd_S\Omega^1(\A)=1$.
Then, for $0\leq p\leq \ell-2$, $j_p$ in \eqref{eq:E} is an isomorphism and $\pd_S\Omega^p(\A)=p$.
Moreover, $\pd_S\Omega^{\ell-1}(\A)=\ell-2$.\qed
\end{thm}

\begin{exa}
The rank-$4$ arrangement $\calA$ of \cite[Example~5.3]{CDFV09} is free outside
points and 
$\pd_S\Omega^1(\calA)=2$. 
Further calculation shows that the map $\bigwedge^2\Omega^1(\calA)\to\Omega^2(\calA)$ is injective, but not an isomorphism.  Moreover, the map $j_p$ need not
be injective in general: here, for example, $\bigwedge^5\Omega^1(\A)$ is a
(nonzero) torsion module supported at the maximal ideal, while $\Omega^5(\log \A)=0$.  
\end{exa}

\begin{dfn}
For $k\in\NN$, we call the divisor $Y$ \emph{$k$-tame} if $\pd_{\calO_{X,x}}\Omega^p(\log Y)_x\le p$ for all $x\in X$ and all $0\le p\le k$.
\end{dfn}

In particular, by Theorem~\ref{th:MSwedges}, an arrangement free outside points is tame if and only if it is $1$-tame. 
Denham and Schulze in \cite[Prop.~2.9]{DS10} give the following variation:

\begin{thm}\label{th:DSwedges}
If the codimension of the non-free locus of $\calA$ is greater than $k$
and $\A$ is $(k-1)$-tame, then the map $j_p$ of \eqref{eq:E} is an isomorphism for $0\leq p< k$.\qed
\end{thm}
As a matter of fact, inspection of the proof of \cite[Prop.~2.9]{DS10}
reveals that $\calA$ need not be an arrangement but can be an
arbitrary $(k-1)$-tame divisor with non-free locus of codimension
$k+1$ or more.

\begin{prb}\label{q-E-p}
Describe in general the modules $E^p_Y$, or at least their vanishing.
\end{prb}

We now prepare the way for a different strengthening (Corollary~\ref{cor:wedges} below) of Theorem~\ref{th:MSwedges}, this
time relaxing the $1$-tameness condition while adhering to the case of
divisors free outside points. 
We first prove general statements on reflexive modules with zero-dimensional non-free locus; these involve the following technical definitions.
The first definition can be found for instance in \cite{Leb75}.

\begin{dfn}
An $R$-module $M$ is an \emph{$r$-syzygy} (of $N$) if there is an exact sequence
\[
0\to M\to P_r\to\dots\to P_1(\to N\to0)
\]
where each $P_i$ is $R$-projective.
On the other hand, $M$ is \emph{$k$-torsion free} if every $R$-regular
sequence of length $\le k$ is also $M$-regular.
\end{dfn}

Being an $r$-syzygy implies $r$-torsion freeness.
The two notions are equivalent, if $\pd_R M<\infty$ (see \cite[p.2]{Leb77}).
The next definition is due to Auslander~\cite{MPS67}:

\begin{dfn}\label{def-spherical}
An $R$-module $M$ is {\em $p$-spherical} if $\pd_R M\leq p$ and $\Ext^i_R(M,R)=0$ for $1\leq i\leq p-1$.
\end{dfn}
For example, $R/I$ is $p$-spherical if and only if $I$ is Cohen-Macaulay of codimension $p$.

\begin{prp}
Let $R$ be an $\ell$-dimensional regular ring. 
Let $M$ be a finitely generated reflexive non-free $R$-module with $c$-dimensional non-free locus, of projective dimension $d=\pd_RM$.  
Then
\begin{equation}\label{pd-ineq}
\pd_RM+\pd_RM^\vee\geq \ell-1-c.
\end{equation}
In case of equality, $M$ and $M^\vee$ are $d$- and $(\ell-d-1-c)$-spherical, respectively.
\end{prp}

\begin{proof}
We apply the spectral sequence of Remark~\ref{rem-Ext-ss}, using our
additional hypotheses.  Since $\Ext^q_R(M,R)$ is Noetherian and
supported for $q>0$ only in dimension at most $c$, then $E^2_{p,q}=0$
unless $q=0$ or $p\leq c-\ell$. Since $M$ is reflexive,
$E^2_{0,0}$ is the abutment $M$.  
So $E^{p,q}_\infty=0$ unless $p=q=0$. 
Of course, $E_2^{p,q}=0$ for
$q>d$. Now suppose $\pd(M^\vee)=a>0$, so $E^{-a,0}_2\not =0$. 
Then all
differentials $d^{-a,0}_k$ out of $E^{-a,0}_k$ end in
$E^{-a-k,k-1}_k$. These target modules are only nonzero if $k-1\leq d$
and simultaneously $-a-k\leq c-\ell$. Summing, the targets are zero
unless (at minimum) $-a-1\leq c-\ell+d$, i.e.~$\ell-c-d-1\leq
a$. It follows that $a<\ell-c-d-1$ is impossible since it would imply
$E^{-a,0}_\infty\not =0$ which we know to be false. 

Suppose now that $a+d=\ell-c-1$.  Then
the argument above shows that $0=E^{-i,0}_\infty = E^{-i,0}_{2}$ for all
$a>i>0$. Hence $M^\vee$ is spherical. By symmetry, the same holds for
$M$. 
Indeed, one obtains in this case a duality:
\[
\Ext_R^a(M^\vee,R)=E^2_{a,0}\cong
E^2_{c-\ell,d}=\Ext_R^{\ell-c}(\Ext_R^d(M,R),R).\qedhere
\]
\end{proof}

A special case of the above proposition is worth singling out.  Suppose that $M$ is a reflexive $R$-module of projective dimension $d$, with a zero-dimensional non-free locus.  Then
\begin{equation}\label{pd-ineq-zerodim}
\pd_RM+\pd_RM^\vee\geq \ell-1,
\end{equation}
and in the case of equality, $M$ and $M^\vee$ are $d$- and $(\ell-d-1)$-spherical, respectively.
%
A partial converse is the following.

\begin{prp}\label{prop:spherical}
Let $R$ be a regular ring of dimension $\ell$.
Let $M$ be a $d$-spherical $R$-module with zero-dimensional non-free locus, where $0<d<\ell$.
Then $M$ is reflexive and $M^\vee$ is $(\ell-d-1)$-spherical with zero-dimensional non-free locus.
Moreover, $M$ and $M^\vee$ are $(\ell-d)$- and $(d+1)$-syzygies, respectively.
\end{prp}

\begin{proof}
By hypothesis, $M$ has a projective resolution of length $d>0$, 
\begin{equation}\label{eq:Mres}
\xymat{
0& M\ar[l] & F_0\ar[l] & F_1\ar[l]&\cdots\ar[l] & F_d\ar[l] & 0.\ar[l]
}
\end{equation}
Since $\Ext^i_R(M,R)=0$ for $0<i<d$ by hypothesis, dualizing gives an exact complex
\[
\xymat{
0\ar[r] & M^\vee\ar[r] & F_0^\vee\ar[r] &\cdots\ar[r] & F_d^\vee\ar[r] & \Ext^d_R(M,R)\ar[r] &0.
}
\]
In particular, $M^\vee$ is a $(d+1)$-syzygy.
By regularity of $R$, $\pd_R\Ext^d_R(M,R)\le\ell$, so this complex extends to a projective resolution
\begin{equation}\label{eq:dec}
\xymat@C-1em@R-15pt{
0\ar[r] & G_{\ell-d-1}\ar[r] & \cdots\ar[r] & G_0\ar[rr]\ar[dr] && F_0^\vee\ar[r] &\cdots\ar[r] & F_d^\vee\ar[r] & \Ext^d_R(M,R)\ar[r] &0,\\
&&&&M^\vee\ar[dr]\ar[ur]\\
&&&0\ar[ur]&&0
}
\end{equation}
and hence $\pd_RM^\vee\le\ell-d-1$.
Since $M$ has zero-dimensional non-free locus, $\Ext^d_R(M,R)$ is supported only at  maximal ideals, and hence
\[
\Ext^i_R(\Ext^d_R(M,R),R)=0,\quad\text{for $i\ne\ell$.}
\]
Thus, dualizing \eqref{eq:dec} and comparing with \eqref{eq:Mres} proves reflexivity of $M$.
By d\'ecalage, \eqref{eq:dec} also shows that $\Ext^i_R(M^\vee,R)=0$ except for $i=0$ and $i=\ell-d-1$, so $M^\vee$ is $(\ell-d-1)$-spherical.
By localization, it clearly has a zero-dimensional non-free locus.
Then $M^\vee$ satisfies the original hypotheses and so $M$ is an $(\ell-d)$-syzygy.
\end{proof}

For convenience, we now state a result of Lebelt \cite[Satz~1]{Leb75}.

\begin{thm}\label{thm:wedges}
Let $1\le p\in\NN$ and assume that the (not necessarily regular) ring
$R$ contains a field $\KK$ with $p<\chr\KK$ or $\chr\KK=0$.
If $M$ is an $R$-module with $\pd_RM=d<\infty$ that is
$(d(p-1)+k)$-torsion free for some $k\geq 0$, then $\bigwedge^p M$ is $k$-torsion free
and $\pd_R\bigwedge^p M\le pd$.

In the graded case, this last inequality becomes an equality.\qed
\end{thm}

\begin{thm}
Suppose $R$ is a regular ring.
Let $M$ be a $d$-spherical $R$-module with zero-dimensional non-free
locus. 
If $\ell=\dim R\le\chr\KK$ or
$\chr\KK=0$ then $\bigwedge^pM$ is reflexive for all $p$ such that
$pd<\ell-1$.
\end{thm}

\begin{proof}
We assume $pd<\ell -1$, which is equivalent to $d(p-1)+2\le \ell-d$.
Since $M$ is $d$-spherical, $d=\pd_RM$. 
By Proposition~\ref{prop:spherical}, $M$ is $(\ell-d)$-torsion
free. By Theorem~\ref{thm:wedges} above, 
$\bigwedge^pM$ is
$2$-torsion free and
\begin{equation}\label{eq:pd-bound}
\pd_R\bigwedge^p M\le pd<\ell-1.
\end{equation} 
Since
$\bigwedge^pM$ is a $2$-syzygy, it is a submodule of a free module and
thus contained in its reflexive hull $M^p=((\bigwedge^pM)^\vee)^\vee$.
Of course, $\pd_R M^p\le\ell-2$. 

Then we have a short exact sequence
\begin{equation}\label{eq-Ep}
0\to\bigwedge^p M\to M^p\to E^p\to 0;
\end{equation}
since $d=\pd_R M^p\leq \ell-2$ as well, applying
$\Ext^\bullet_R(-,R)$ shows that $\Ext^\ell_R(E^p,R)=0$.
On the other hand, $M$ has 
zero-dimensional non-free locus, hence $E^p$ is finite length, which
means $\Ext^q_R(E^p,R)=0$ for $q<\ell$.
It follows that
$E^p=0$, hence $\bigwedge^pM\cong M^p$ by \eqref{eq-Ep}, which is reflexive by definition.
\end{proof}

\begin{cor}\label{cor:wedges}
Let $Y\subseteq X$ be a divisor in an $\ell$-dimensional complex manifold or smooth algebraic variety, or let $Y=\A$ be a hyperplane arrangement in $X=\AA_\KK^\ell$ with $\ell\le\chr\KK$ or $\chr\KK=0$. 
Assume that $Y$ is free outside points and that $\Omega^1(\log Y)_x$ is $d$-spherical for all $x\in X$. 
Then $\Omega^p(\log Y)=\bigwedge^p\Omega^1(\log Y)$ for $pd<\ell-1$.\qed
\end{cor}


\begin{rmk}
Suppose $Y$ is a homogeneous hypersurface in affine space. If $Y$ is
not a free divisor, a natural question is whether any of the modules
$\Omega^i(\log Y)$ can be free. In dimension three or less, this is
impossible because of duality. In dimensions $4$ and $5$, the question
boils down to asking whether $\Omega^2(\log Y)$ can be free while
$\Omega^1(\log Y)$ is not. One sees this to be impossible as follows:
$\Omega^2(\log Y)$
decomposes into $\Omega^1_0(\log Y)\oplus \Omega^2_0(\log Y)$ so that
freeness of $\Omega^2(\log Y)$ implies freeness of $\Omega^1_0(\log Y)$
and hence of $\Omega^1(\log Y)$. Thus, the first open case appears for
$p=3$ and $\ell=6$.
\end{rmk}

\section{Generic arrangements}\label{sec-generic}

We consider here generic arrangements $\calA$. 
Again, we assume that the characteristic of the base field is good (Definition~\ref{dfn:goodp}).
The purpose of this section is to prove Theorem~\ref{th:generic} below. 

Recall from \cite[Def.~2.4]{DS10} that the module of \emph{relative differential $p$-forms along $\A$} is the kernel
$\Omega_0^p(\calA)$ of contraction with the Euler vector field $\chi$,
\[
\chi\colon\Omega^p(\calA)\to\Omega^{p-1}(\calA).
\]
These modules are non-zero only for $0\le p\le\ell-1$;
for $p=0$ and $p=\ell-1$ they are free  of rank $1$. 
By \cite[(2.4)]{DS10}, we can identify 
\[
\Omega^p(\A)\cong\Omega^{p-1}_0(\A)\oplus\Omega^p_0(\A).
\]
So both tameness and being free outside points descends from $\Omega^1(\calA)$
to $\Omega^{\bullet}_0(\A)$.  In particular, $\Omega_0^p(\A)$
has zero-dimensional non-free locus for $0<p<\ell-1$ if $\A$ is
generic.  Recall also the module $D^0_p(\A)=D_p(\A)/\chi\wedge
D_{p-1}(\A)$ of \emph{relative logarithmic $p$-derivations along
  $\A$} defined in \cite[Def.~3.4]{DS10}, which has the analogous
property
\[
D_p(\A)\cong D_{p-1}^0(\A)\oplus D_p^0(\A).
\]
By\cite[Prop.~3.5]{DS10}, $D^0_p(\A)$ and $\Omega^p_0(\A)$ are mutually $S$-dual.

\begin{thm}\label{th:generic}
Let $\calA$ be a generic non-Boolean rank-$\ell$ arrangement.
Then $\Omega_0^p(\calA)$ is $p$-spherical, for $0<p<\ell-1$, and $\Ext_S^p(\Omega^p_0(\calA),S)$ is Artinian of length $n-1\choose \ell$ where $n=\abs{\calA}$.
\end{thm}

The proof is by induction on $n$. 
We first establish some lemmas in order to proceed by means of a deletion-restriction argument.

\begin{dfn}
For an arrangement $\calA$ and a hyperplane $H\in\calA$ we denote by
$\calA'$ the arrangement $\calA\minus\{H\}$ and by $\calA''$ the
arrangement induced by $\calA'$ on $H$.
\end{dfn}

Now fix $H\in\A$ with defining equation $\alpha_H$.
From \cite[Prop.~4.45]{OTbook} and Ziegler \cite[Cor.~4.5]{Zie89}, there are sequences of $S$-modules
\begin{gather}
\label{eq:Dseq}\xymat{
0\ar[r] & D(\calA')(-1)\ar[r]^-{\alpha_H} & D(\calA)\ar[r] & D(\calA'')\ar@{-->}[r] & 0},\\
\label{eq:Omseq}\xymat{
0\ar[r] & \Omega^p(\calA)(-1)\ar[r]^-{\alpha_H} & \Omega^p(\calA')\ar[r] & \Omega^p(\calA'')\ar@{-->}[r] & 0},
\end{gather}
which are exact except possibly at $D(\calA'')$ and $\Omega^p(\calA')$.
Although Ziegler formulates this result only for $p=1$, it holds true in general.

\begin{lem}\label{lem:wiens}
If $\calA$ is non-Boolean generic then there are short exact sequences
\begin{gather}
\label{eq:wiensDer}\xymat@R=0pt{
0\ar[r] & D^0(\calA')(-1)\ar[r]^-{\alpha_H} & D^0(\calA)\ar[r] & D^0(\calA'')\ar[r] & 0,}\\
\label{eq:wiensOmega}\xymat@R=0pt{
0\ar[r] & \Omega_0^{\ell-2}(\calA')\ar[r]^-{\alpha_H} & \Omega_0^{\ell-2}(\calA)\ar[r] & 
\Omega_0^{\ell-3}(\calA'')\ar[r] & 0.
}\qed
\end{gather}
\end{lem}

\begin{proof}
By Wiens~\cite[Thm.~3.4]{Wi01}, \eqref{eq:Dseq} is exact at $D(\calA'')$.
By dividing out $S\chi$, this proves exactness of the first sequence.
Then the second sequence is obtained via the identification $D^0(\calA)=\Omega_0^{\ell-2}(\calA)(\ell-n)$ from \cite[Prop.~3.7.(3)]{DS10}.
\end{proof}

\begin{lem}\label{lem:shift}
Suppose $\alpha$ is a non-zerodivisor in a ring $R$.
If $M$ is a module over $R''=R/R\alpha$, then
\begin{equation}\label{eq:shift}
\Ext_{R''}^q(M,R'')
\cong\Ext_R^{q+1}(M,R)
\end{equation}
for $q\geq0$, and $\Ext_R^0(M,R)=0$. If $R$ is graded and $\alpha$ 
homogeneous, the isomorphism becomes graded after twisting the left
hand side by $\deg(\alpha)$.
\end{lem}

\begin{proof}
As an $R$-module, $R''$ has a free resolution
\[
\xymat{
0\ar[r] & R\ar[r]^-\alpha & R \ar[r] & R''\ar[r] & 0.
}
\]
So $\Ext^q_R(R'',R)=0$ unless $q=1$, in which case we get $\Ext_R^1(R'',R)=R''$. 
The change-of-rings spectral sequence 
\[
E^{p,q}_2=\Ext_{R''}^p(M,\Ext^q_R(R'',R))\Rightarrow\Ext_R^{p+q}(M,R)
\]
has only one nonzero row, $q=1$, from which the result follows.
\end{proof}

\begin{lem}\label{lem:Opseq}
If $\calA$ is generic non-Boolean of rank $\ell\ge4$ then there is a short exact sequence 
\begin{equation}\label{eq:Opseq}
\xymat{
0\ar[r] & \Omega_0^p(\calA)(-1)\ar[r]^-{\alpha_H} & \Omega_0^p(\calA')\ar[r] & \Omega_0^p(\calA'')\ar[r] & 0
}
\end{equation}
for $0\leq p\leq \ell-3$.  
\end{lem}

\begin{proof}
We first establish the case $p=1$. 
From Ziegler's presentation of $\Omega^1(\calA)$ of a generic arrangement \cite[Cor.~7.7]{Zie89}, we have $\pd_S\Omega^1_0(\calA)=\pd_S\Omega^1(\calA)=1$.  So
$\Omega^1_0(\calA)$ is $1$-spherical with zero-dimensional non-free
locus.  It follows that its dual, $D^0(\calA)$, is
$(\ell-2)$-spherical, by Proposition~\ref{prop:spherical}.  Since $\ell-2>1$,
we have $\Ext_R^1(D^0(\calA),R)=0$.  By
Lemma~\ref{lem:shift}, we have
\[
\Ext_R^1(D^0(\A''),R)=\Hom_{R''}(D^0(\A''),R'')(1),\quad R''=R/R\alpha_H.
\]
Recall that $D^0_\bullet(\A)$ and $\Omega^\bullet_0(\A)$ are 
mutually $R$-dual by \cite[Prop.~3.5]{DS10}.
Lemma~\ref{lem:mapseq} below gives exactness of \eqref{eq:Opseq} for $p=1$.

For general $p$, we proceed as follows.  As both multiplication by
$\alpha_H$ and restriction to $H$ commute with contraction against
$\chi$, passing to the kernels of $\chi$ in \eqref{eq:Omseq} gives
exactness of \eqref{eq:Opseq}, except at the right module.

Since $\A'$ and $\A''$ are again generic, they are both free outside points
and tame.  So we have $\bigwedge^p\Omega_0^1(\calA')\cong\Omega^p_0(\calA')$ for $p\leq \ell-2$ and $\bigwedge^p\Omega_0^1(\calA'')\cong \Omega^p_0(\calA'')$ for $p\leq\ell-3$, by \cite[Prop.~2.9]{DS10}.
%
Now by surjectivity in case $p=1$ and right-exactness of $\bigwedge^p$, \eqref{eq:Opseq} is exact on the right as well.
\end{proof}


\begin{lem}\label{lem:mapseq}
The sequence \eqref{eq:Opseq} for $p=1$ is obtained by dualizing \eqref{eq:wiensDer} and applying \eqref{eq:shift} for $M=D^0(\A'')$ and $q=0$. 
\end{lem}
\begin{proof}
Ignoring degrees and dropping $0$-indices, we need to show that the
restriction map $\Omega^1(\A')\to\Omega^1(\A'')$ from \eqref{eq:Omseq}
coincides with the composition of the connecting homomorphism 
\[
\Hom_S(D(\A'),S)\to\Ext^1_{S}(D(\A''),S)\cong\Hom_{S''}(D(\A''),S''),
\]
obtained from dualizing \eqref{eq:Dseq}, 
with the isomorphism
\[
\Ext^1_{S}(D(\A''),S)\cong\Hom_{S''}(D(\A''),S''),
\]
obtained from \eqref{eq:shift} with $M=D(\A'')$ and $q=0$. 

The claim
is trivially true outside the arrangement $\calA''$.
The source and target of these two maps are naturally isomorphic and
(reflexive, hence,)  normal. 
Therefore it suffices to prove the statement locally at a generic
point of $\A''$.  

The arrangement $\A$ is generic; thus, in a generic point of
$\calA''$, $\calA$ is the
product of a Boolean $2$-arrangement with $\CC^{\ell-2}$. In
particular,  $D(\A'')_x$, $D(\A')_x$ and $D(\A)_x$ are free in such a
point.  Then \eqref{eq:Dseq} is a free resolution of $D(\A'')_x$ and
the claim becomes an exercise in homological algebra.

Abbreviate $A=D(\A')_x$, $B=D(\A)_x$, and $C=D(\A'')_x$; then the
complex \eqref{eq:Dseq} becomes
\[
\xymat{
0\ar[r] & A\ar[r]^-{\alpha} & B\ar[r] & C\ar[r] & 0,
}
\]
and it has a free resolution
\[
\xymat{
0\ar[r] & A\ar[r]^-{(\alpha,0)} & \alpha A\oplus B\ar[r]^-{(0,\id_B)} & B\ar[r] & 0\\
0\ar[r] & 0\ar[u]\ar[r] & A\ar[u]_-{(\alpha,-\alpha)}\ar[r]^-= & A\ar[r]\ar[u]_-{-\alpha} & 0\\
& 0\ar[u] & 0\ar[u] & 0\ar[u].
}
\]
After dualizing, the connecting homomorphism $A^\vee\to A^\vee/\alpha^\vee B^\vee$ shows up in the $E^2$-page of the spectral sequence of the double complex.
It is induced by $(\alpha^{-1}\circ\alpha)^\vee=\id_{A^\vee}$, so it coincides with the natural restriction map.
\end{proof}

Since our calculation does not depend on the choice of equations for the generic
arrangement, let $\calA_{n,\ell}$ denote a generic arrangement of $n$ hyperplanes with
rank $\ell$.  

The base case of our induction argument will be the following.

\begin{lem}\label{lem:rank3}
For $\calA=\calA_{n,3}$ with $n>3$, $\Ext^1_S(\Omega^1_0(\calA),S)$ is an Artinian module with Hilbert series
\[
q_n(t)=\big((n-3)t^{-1}+(1-n)+(n-1)t^{n-3}+(3-n)t^{n-2}\big)/(1-t)^3.
\]
\end{lem}

\begin{proof}
By \cite[Cor.~7.7]{Zie89}, one has a free resolution
\begin{equation}\label{eq:resO}
\xymat{
0\ar[r] & S(-1)^{n-3}\ar[r] & S^{n-1}\ar[r] & \Omega_0^1(\calA)\ar[r] & 0.
}
\end{equation}
On the other hand, following \cite{Yu91}, the module of derivations for the generic arrangement has a graded free resolution
\[
\xymat{
0\ar[r] & S(2-n)^{n-3}\ar[r] & S(3-n)^{n-1}\ar[r] &
D^0(\calA)\ar[r] & 0,
}
\]
so $h(D^0(\A),t)=((n-1)t^{n-3}-(n-3)t^{n-2})/(1-t)^3$.
Now we dualize \eqref{eq:resO} to obtain
\[
\xymat{
0 & \Ext^1_S(\Omega_0^1(\calA),S)\ar[l] & S^{n-3}(1)\ar[l] & S^{n-1}\ar[l] & D^0(\calA)\ar[l] & 0,\ar[l]
}
\]
and take the Euler characteristic to compute the Hilbert series of $\Ext^1_S(\Omega_0^1(\calA),S)$.
\end{proof}

\begin{rmk}
The previous result can be expressed more concisely in terms of a 
generating function.  One can calculate that
\begin{eqnarray*}
\sum_{n\geq4}q_n(t)s^n &=& \frac1{(1-t)^3}\sum_{n\geq4}
\big((n-3)t^{-1}+(1-n)+(n-1)t^{n-3}+(3-n)t^{n-2}\big)s^n\\
&=&\frac1{(1-t)^3}\Big(
\frac{s^4t^{-1}}{(1-s)^2}+\frac{(2s-3)s^4}{(1-s)^2}+
\frac{(3-2st)(s^4t)}{(1-st)^2}-\frac{s^4t^2}{(1-st)^2}\Big)\\
&=&t^{-1}\frac{s^4}{(1-s)^2(1-st)^2}.
\end{eqnarray*}

In particular, by setting $t=1$ in the expression above, we obtain
$s^4/(1-s)^4$, and we find that $\Ext^1_S(\Omega^1_0(\calA),S)$ is a
module of length equal to $n-1\choose 3$.
\end{rmk}

\begin{proof}[Proof of Theorem~\ref{th:generic}]
Suppose that the conclusion is known for all arrangements of fewer
than $n$ hyperplanes.  We prove the claim for $\calA_{n,\ell}$ where
$n>\ell$.  We will assume $\ell\geq4$, since the rank $3$ case is
covered by Lemma~\ref{lem:rank3}.  By hypothesis, $\Omega^p_0(\calA')$
and $\Omega^p_0(\calA'')$ are $p$-spherical for $1\leq p\leq \ell-3$.
Now apply $\Ext^\bullet_R(-,R)$ to the short exact sequence \eqref{eq:Opseq}.
By Lemma~\ref{lem:shift}, we may replace $\Ext^{q+1}_R(\Omega^p_0(\A''),R)$
by $\Ext^q_{R''}(\Omega^p_0(\A''),R'')$ in the long exact sequence.
Then the $p$-spherical condition, together with the exactness of
\eqref{eq:wiensDer} for $i=0$, implies that the long exact sequence
breaks up as
\begin{equation}\label{eq:indp}
\xymat@C=15pt{
0\ar[r] & \Ext^i_S(\Omega_0^p(\calA'),S)(-1)\ar[r] &
\Ext^i_S(\Omega_0^p(\calA),S)\ar[r] &
\Ext^i_{S''}(\Omega_0^p(\calA''),S'')\ar[r] & 0,
}
\end{equation}
for all $i\geq0$ and for $0\leq p\leq \ell-3$.

On the other hand, for $p=\ell-2$, note $p\geq2$.  Consider the long exact sequence obtained from the dual of \eqref{eq:wiensOmega}. 
By Lemma~\ref{lem:mapseq}, it begins with the sequence \eqref{eq:Opseq}, which is exact from Lemma~\ref{lem:Opseq}.
It follows that $\Ext^1_S(\Omega^{\ell-2}_0(\calA),S)=0$.

For $i\geq2$, the assumption that $\Omega^{\ell-2}_0(\calA')$ and $\Omega^{\ell-3}_0(\calA'')$ are $(\ell-2)$-and $(\ell-3)$-spherical, respectively, gives a short exact sequence
\begin{equation}\label{eq:indptop}
\xymat@C=10pt{
0\ar[r] & \Ext^{i-1}_{S''}(\Omega_0^{\ell-3}(\calA''),S'')(1)\ar[r] &
\Ext^i_S(\Omega_0^{\ell-2}(\calA),S)\ar[r] &
\Ext^i_{S}(\Omega_0^{\ell-2}(\calA'),S)\ar[r] & 0.
}
\end{equation}

It follows that $\Omega^p_0(\calA)$ is $p$-spherical for $1\leq p\leq \ell-3$ by 
\eqref{eq:indp}, and for $p=\ell-2$ by \eqref{eq:indptop}.  In either 
case, the length of $\Ext^p_S(\Omega^p_0(\calA),S)$ can be computed using
the induction hypothesis and \eqref{eq:indp} or \eqref{eq:indptop} to be 
\[
{n-2\choose\ell}+{n-2\choose\ell-1} = {n-1\choose\ell}.
\]
\end{proof}

If one repeats the argument, it is possible to obtain the Hilbert
series of each $\Ext^p_S(\Omega_0^p(\calA),S)$, instead of just its
length.  Given the number of parameters involved, we present this
refinement to Theorem~\ref{th:generic} separately.

\begin{dfn}
Let $q(\ell,n,p;t)$ denote the Hilbert series of 
$\Ext_S^p(\Omega^p_0(\calA_{n,\ell}),S)$, for $3\leq \ell\leq n$ and
$1\leq p\leq \ell-2$.  Let $Q(\ell,p;s,t)$ be the generating function
\[
Q(\ell,p;s,t)=\sum_{n\geq3} q(\ell,n,p;t) s^n,
\]
and write $Q(\ell,p)$ for short.
\end{dfn}

We saw, above, that 
\begin{equation}\label{eq:Q31}
Q(3,1)=s^4/\big(t(1-s)^2(1-st)^2\big).
\end{equation}

\begin{lem}
The generating functions $Q(\ell,p)$ satisfy
\[
Q(\ell,p)=\begin{cases}
\frac{s}{1-st}Q(\ell-1,p)&\text{ for $1\leq p\leq \ell-3$;}\\
\frac{s}{t(1-s)} Q(\ell-1,\ell-3) & \text{ for $p=\ell-2$}.
\end{cases}
\]
\end{lem}

\begin{proof}
The short exact sequence \eqref{eq:indp}, for $i=p$, gives
\[
q(\ell,n,p;t)= tq(\ell,n-1,p;t)+q(\ell-1,n-1,p;t)
\]
as long as $1\leq p\leq \ell-3$.  Writing the sum of both sides gives the equation
$Q(\ell,p;s,t)=st Q(\ell,p;s,t)+sQ(\ell-1,p;s,t)$, from which we
obtain the formula above.  The case $p=\ell-2$ is similar, using
\eqref{eq:indptop} instead.
\end{proof}

By applying this result recursively, beginning with \eqref{eq:Q31}, we obtain the expression:

\begin{lem}\label{lem:Qseries}
For all $\ell\geq 3$ and $1\leq p\leq \ell-2$,
\begin{equation}\label{eq:Qseries}
Q(\ell,p;s,t)=t^{-p}\frac{s^{\ell+1}}{(1-s)^{p+1}(1-st)^{\ell-p}}.
\end{equation}
\end{lem}

Continuing further, let
\[
P(p;s,t,u)=\sum_{\ell\geq p+2} Q(\ell,p;s,t)u^\ell.
\]
Then the geometric series formula, using \eqref{eq:Qseries}, simplifies to the following:
\begin{equation}\label{eq:biggenfcn}
P(p;s,t,u)=\frac{s^3 u^2}{(1-s)(1-st)(1-s(t+u))}
\Big(\frac{su}{t(1-s)}\Big)^p,
\end{equation}
for $p\geq 1$.  Finally, we can form a four-variable generating function
\[
T(s,t,u,v)=\sum_{p\geq1} P(p;s,t,u)v^p.
\]
By construction,
\[
T(s,t,u,v)=\sum_{n,\ell,p} 
h(\Ext_S^p(\Omega_0^p(\calA_{n,\ell}),S),t) s^n u^\ell v^p.
\]
Since each module is Artinian, $T\in\Z[t,t^{-1}][[s,u,v]]$ (and, in particular,
our rational function calculations do lead to a valid expression for a 
formal power series).
By summing the expression \eqref{eq:biggenfcn}, we obtain the following:
\begin{cor}
The Hilbert series of the module $\Ext_S^p(\Omega_0^p(\calA_{n,\ell}),S)$ is
the coefficient of $s^nu^\ell v^p$ in
\[
T(s,t,u,v)=\frac{s^4u^3v}{(1-s)(1-st)(1-st-su)(t-st-suv)},
\]
for all $1\leq p\leq\ell-2\leq n-2$.
\end{cor}

\subsection*{Acknowledgments}

We would like to express our gratitude to the American Institute of
Mathematics in Palo Alto for its support and hospitality during our
three SQuaRE meetings, out of which this work grew.


\bibliographystyle{amsplain}

\begin{thebibliography}{10}

\bibitem{Bor87}
A.~Borel, P.-P. Grivel, B.~Kaup, A.~Haefliger, B.~Malgrange, and F.~Ehlers,
  \emph{Algebraic {$D$}-modules}, Perspectives in Mathematics, vol.~2, Academic
  Press Inc., Boston, MA, 1987. \MRh{882000}{89g:32014}

\bibitem{Bri71}
Egbert~Brieskorn, \emph{Singular elements of semi-simple algebraic groups}, Actes
  du {C}ongr\`es {I}nternational des {M}ath\'ematiciens ({N}ice, 1970), {T}ome
  2, Gauthier-Villars, Paris, 1971, pp.~279--284. \MRh{0437798}{55 \#10720}

\bibitem{Bru84}
James~W. Bruce, \emph{Functions on discriminants}, J. London Math. Soc. (2)
  \textbf{30} (1984), no.~3, 551--567. \MRh{810963}{87e:58028}

\bibitem{Cal99}
Francisco~J. Calder{\'o}n-Moreno, \emph{Logarithmic differential operators and
  logarithmic de {R}ham complexes relative to a free divisor}, Ann. Sci.
  \'Ecole Norm. Sup. (4) \textbf{32} (1999), no.~5, 701--714. 
\MRh{1710757}{2000g:32010}

\bibitem{CMNC02}
Francisco~J. Calder{\'o}n~Moreno, David Mond, Luis Narv{\'a}ez~Macarro, and
  Francisco~J. Castro~Jim{\'e}nez, \emph{Logarithmic cohomology of the
  complement of a plane curve}, Comment. Math. Helv. \textbf{77} (2002), no.~1,
  24--38. \MRh{1898392}{2003e:32047}

\bibitem{CNM96}
Francisco~J. Castro-Jim{\'e}nez, Luis Narv{\'a}ez-Macarro, and David Mond,
  \emph{Cohomology of the complement of a free divisor}, Trans. Amer. Math.
  Soc. \textbf{348} (1996), no.~8, 3037--3049. \MRh{1363009}{96k:32072}

\bibitem{CDFV09}
D.~Cohen, G.~Denham, M.~Falk, and A.~Varchenko, \emph{Critical points and
  resonance of hyperplane arrangements}, Canad. J. Math. \textbf{63} (2011),
  no.~5, 1038--1057. \MR{2866070}

\bibitem{dGMS09}
Ignacio de~Gregorio, David Mond, and Christian Sevenheck, \emph{Linear free
  divisors and {F}robenius manifolds}, Compos. Math. \textbf{145} (2009),
  no.~5, 1305--1350. \MR{2551998}

\bibitem{DS10}
Graham Denham and Mathias Schulze, \emph{Complexes, duality and {C}hern classes
  of logarithmic forms along hyperplane arrangements}, Arrangements of
  hyperplanes---{S}apporo 2009, Adv. Stud. Pure Math., vol.~62, Math. Soc.
  Japan, Tokyo, 2012, pp.~27--57. \MR{2933791}

\bibitem{ER93}
Paul~H. Edelman and Victor Reiner, \emph{A counterexample to {O}rlik's
  conjecture}, Proc. Amer. Math. Soc. \textbf{118} (1993), no.~3, 927--929.
  \MRh{1134624}{93i:52021}

\bibitem{EHV92}
David Eisenbud, Craig Huneke, and Wolmer Vasconcelos, \emph{Direct methods for
  primary decomposition}, Invent. Math. \textbf{110} (1992), no.~2, 207--235.
  \MRh{1185582}{93j:13032}

\bibitem{GMNS09}
Michel Granger, David Mond, Alicia Nieto-Reyes, and Mathias Schulze,
  \emph{Linear free divisors and the global logarithmic comparison theorem},
  Ann. Inst. Fourier (Grenoble) \textbf{59} (2009), no.~2, 811--850.
  \MR{2521436}

\bibitem{GMS09}
Michel Granger, David Mond, and Mathias Schulze, \emph{Free divisors in
  prehomogeneous vector spaces}, Proc. Lond. Math. Soc. (3) \textbf{102}
  (2011), no.~5, 923--950. \MR{2795728}

\bibitem{GS06}
Michel Granger and Mathias Schulze, \emph{On the formal structure of
  logarithmic vector fields}, Compos. Math. \textbf{142} (2006), no.~3,
  765--778. \MRh{2231201}{2007e:32037}

\bibitem{GS10}
\bysame, \emph{On the symmetry of $b$-functions of linear free divisors}, Publ.
  RIMS \textbf{46} (2010), no.~3, 479--506. \MRh{2760735}{2011k:14014}

\bibitem{Leb75}
Karsten Lebelt, \emph{Zur homologischen {D}imension \"ausserer {P}otenzen von
  {M}oduln}, Arch. Math. (Basel) \textbf{26} (1975), no.~6, 595--601.
  \MRh{0396534}{53 \#397}

\bibitem{Leb77}
\bysame, \emph{Freie {A}ufl\"osungen \"ausserer {P}otenzen}, Manuscripta Math.
  \textbf{21} (1977), no.~4, 341--355. \MRh{0450253}{56 \#8549}

\bibitem{Loo84}
E.~J.~N. Looijenga, \emph{Isolated singular points on complete intersections},
  London Mathematical Society Lecture Note Series, vol.~77, Cambridge
  University Press, Cambridge, 1984. \MRh{747303}{86a:32021}

\bibitem{MPS67}
Marguerite Mangeney, Christian Peskine, and Lucien Szpiro, \emph{Anneaux de
  {G}orenstein, et torsion en alg\`ebre commutative}, S\'eminaire d'Alg\`ebre
  Commutative dirig\'e par Pierre Samuel, 1966/67. Texte r\'edig\'e, d'apr\`es
  des expos\'es de Maurice Auslander, Marquerite Mangeney, Christian Peskine et
  Lucien Szpiro. \'Ecole Normale Sup\'erieure de Jeunes Filles, Secr\'etariat
  math\'ematique, Paris, 1967. \MRh{0225844}{37 \#1435}

\bibitem{Mat89}
Hideyuki Matsumura, \emph{Commutative ring theory}, second ed., Cambridge
  Studies in Advanced Mathematics, vol.~8, Cambridge University Press,
  Cambridge, 1989, Translated from the Japanese by M. Reid. \MRh{1011461}{90i:13001}

\bibitem{MS10}
David Mond and Mathias Schulze, \emph{Adjoint divisors and free divisors},
  arXiv.org \textbf{math.AG} (2010), no.~1001.1095, Submitted.

\bibitem{MuSc01}
Mircea Musta{\c{t}}{\v{a}} and Henry~K. Schenck, \emph{The module of
  logarithmic {$p$}-forms of a locally free arrangement}, J. Algebra
  \textbf{241} (2001), no.~2, 699--719. \MRh{1843320}{2002c:32047}

\bibitem{OTbook}
Peter Orlik and Hiroaki Terao, \emph{Arrangements of hyperplanes}, Grundlehren
  der Mathematischen Wissenschaften [Fundamental Principles of Mathematical
  Sciences], vol. 300, Springer-Verlag, Berlin, 1992. \MRh{1217488}{94e:52014}

\bibitem{Ro62}
Jan-Erik Roos, \emph{Bidualit\'e et structure des foncteurs d\'eriv\'es de
  {$\varinjlim$} dans la cat\'egorie des modules sur un anneau r\'egulier}, C.
  R. Acad. Sci. Paris \textbf{254} (1962), 1556--1558. \MRh{0136640}{25 \#106b}

\bibitem{RT91}
Lauren~L. Rose and Hiroaki Terao, \emph{A free resolution of the module of
  logarithmic forms of a generic arrangement}, J. Algebra \textbf{136} (1991),
  no.~2, 376--400. \MRh{1089305}{93h:32048}

\bibitem{Sai71}
Kyoji Saito, \emph{Quasihomogene isolierte {S}ingularit\"aten von
  {H}yperfl\"achen}, Invent. Math. \textbf{14} (1971), 123--142. \MR{0294699}{45 \#3767}

\bibitem{Sai80}
\bysame, \emph{Theory of logarithmic differential forms and logarithmic vector
  fields}, J. Fac. Sci. Univ. Tokyo Sect. IA Math. \textbf{27} (1980), no.~2,
  265--291. \MR{586450}{83h:32023}

\bibitem{Sev09}
Christian Sevenheck, \emph{Bernstein polynomials and spectral numbers for
  linear free divisors}, Ann. Inst. Fourier (Grenoble) \textbf{61} (2011),
  no.~1, 379--400. \MR{2828135}

\bibitem{Sim06}
Aron Simis, \emph{Differential idealizers and algebraic free divisors},
  Commutative algebra, Lect. Notes Pure Appl. Math., vol. 244, Chapman \&
  Hall/CRC, Boca Raton, FL, 2006, pp.~211--226. \MRh{2184799}{2006k:13055}

\bibitem{Slo80}
Peter Slodowy, \emph{Simple singularities and simple algebraic groups}, Lecture
  Notes in Mathematics, vol. 815, Springer, Berlin, 1980. \MRh{584445}{82g:14037}

\bibitem{ST87}
Louis~Solomon and Hiroaki~Terao, \emph{A formula for the characteristic polynomial of
  an arrangement}, Adv. in Math. \textbf{64} (1987), no.~3, 305--325.
  \MR{888631}{88m:32022}

\bibitem{Ter80}
Hiroaki Terao, \emph{Free arrangements of hyperplanes and unitary reflection
  groups}, Proc. Japan Acad. Ser. A Math. Sci. \textbf{56} (1980), no.~8,
  389--392. \MRh{596011}{82e:32018a}

\bibitem{Ter83}
\bysame, \emph{Discriminant of a holomorphic map and logarithmic vector
  fields}, J. Fac. Sci. Univ. Tokyo Sect. IA Math. \textbf{30} (1983), no.~2,
  379--391. \MRh{722502}{85d:32027}

\bibitem{vSt95}
Duco~van Straten, \emph{A note on the discriminant of a space curve}, Manuscripta
  Math. \textbf{87} (1995), no.~2, 167--177. \MRh{1334939}{96e:32032}

\bibitem{Wi01}
Jonathan Wiens, \emph{The module of derivations for an arrangement of
  subspaces}, Pacific J. Math. \textbf{198} (2001), no.~2, 501--512.
  \MRh{1835521}{2002d:14090}

\bibitem{WiYu97}
Jonathan Wiens and Sergey Yuzvinsky, \emph{De {R}ham cohomology of logarithmic
  forms on arrangements of hyperplanes}, Trans. Amer. Math. Soc. \textbf{349}
  (1997), no.~4, 1653--1662. \MRh{1407505}{97h:52013}

\bibitem{Yu91}
Sergey Yuzvinsky, \emph{A free resolution of the module of derivations for
  generic arrangements}, J. Algebra \textbf{136} (1991), no.~2, 432--438.
  \MRh{1089307}{92b:52026}

\bibitem{Zak83}
V.~M. Zakalyukin, \emph{Reconstructions of fronts and caustics depending on a
  parameter, and versality of mappings}, Current problems in mathematics,
  {V}ol. 22, Itogi Nauki i Tekhniki, Akad. Nauk SSSR Vsesoyuz. Inst. Nauchn. i
  Tekhn. Inform., Moscow, 1983, pp.~56--93. \MRh{735440}{85h:58029}

\bibitem{Zie89}
G{\"u}nter~M. Ziegler, \emph{Combinatorial construction of logarithmic
  differential forms}, Adv. Math. \textbf{76} (1989), no.~1, 116--154.
  \MRh{1004488}{90j:32016}

\end{thebibliography}

\newcommand{\arxiv}[1]
{\texttt{\href{http://arxiv.org/abs/#1}{arXiv:#1}}}

\renewcommand{\MR}[1]
{\href{http://www.ams.org/mathscinet-getitem?mr=#1}{MR#1}}

\newcommand{\MRh}[2]
{\href{http://www.ams.org/mathscinet-getitem?mr=#1}{MR#1 (#2)}}


\providecommand{\bysame}{\leavevmode\hbox to3em{\hrulefill}\thinspace}
\providecommand{\MRhref}[2]{%
  \href{http://www.ams.org/mathscinet-getitem?mr=#1}{#2}
}
\providecommand{\href}[2]{#2}


\end{document}